\documentclass[11pt]{article}

\usepackage{amsmath,amsthm,verbatim,amssymb,amsfonts,amscd, graphicx}
\usepackage{graphics}
\theoremstyle{plain}
\newtheorem{theorem}{Theorem}

\theoremstyle{definition}
\newtheorem{definition}{Definition}

\begin{document}

\title{ Harmonic-Counting Measures and Spectral Theory of  Lens Spaces}
\author{H. Mohades  B.Honari}

\maketitle
\begin{abstract} In this article, associated with each lattice $T\subseteq \mathbb{Z}^n$ the concept of a harmonic-counting measure $\nu_T$ on a sphere $S^{n-1}$ is introduced  and it is applied to determine the asymptotic behavior of the eigenfunctions of the Laplace-Beltrami operator on a lens space. In fact, the asymptotic behavior of the cardinality of the set  of independent eigenfunctions associated with the elements of $T$ which lie in a cone is determined when $T$ is the  lattice of a lens space. 


\end{abstract}

\section{Introduction}

\begin{par}
  \end{par}
Counting the number of points of a lattice in a convex body is well studied by many mathematicians including Minkowski, Ehrhart and Stanley. 
The asymptotic behavior of such counting functions leads to the definition of lattice-counting-measures on the sphere $S^{n-1}$ \cite{du, eh, st}. In this paper we define the parallel notion of a harmonic counting measure. Let $H$ be the vector  space of     $2n$-variable  harmonic  polynomials    that are invariant under the action $(z_1,z_2,...,z_n)\rightarrow(e^{2\pi\frac{p_1}{q}}z_1,...,e^{2\pi\frac{p_n}{q}}z_n)$ where  $q$ and the $p_i$ belong to $\mathbb{N}$ and we identify $\mathbb{R}^{2n}$ with $\mathbb{C}^{n}$. $H$ can be decomposed in a natural way into the  vector spaces $H_{s,{(a_1,...,a_n)}}$, which consist of  harmonic homogeneous polynomials of degree $s$ associated with $(a_1,...,a_n)\in T=\{(a_1,...,a_n) \in (\mathbb{N}\cup \{ 0 \})^n|\sum_{j=1}^n a_jp_j\equiv 0\hspace{0.2cm}(\mathrm{mod}\hspace{0.1cm}q)\}$  \cite{HO,HS}. 
 Asymptotic behavior of  $F_{T\cap K}(t)= \sum_{s=0}^t\sum_{x\in T\cap K\,}dimH_{s,x}$   leads to the definition of a harmonic counting measure where $ K$ is an $l_1$-spherical cone. 
The natural relation between lattices, harmonic polynomials and the topological lens spaces is introduced in \cite{ HS, ik11, ik}.
In these articles it is shown that some conditions on lattices determine the isospectrality and isometry of lens spaces.  Also, it is shown that harmonic homogeneous  polynomials that are invariant under the action of homotopy group of a lens space determine  eigenspaces of the Laplace-Beltrami operator. Succinctly, in the present article we define harmonic counting measures and determine their relation to lattice-counting-measures and  isospectrality of lens spaces. 
 \\   

\section{Preliminaries on lens spaces}
\subsection{Lattices }
A lattice $T$ is  a subgroup of the group $\mathbb{Z}^n$. $T$  is of rank $n$ if $T\bigotimes \mathbb{R}=\mathbb{R}^n$. 
\begin{definition}A preliminary lattice group $T$ is defined as $$T=\{(a_1,...,a_n) \in \mathbb{Z}^n|\sum_{j=1}^n a_jp_j\equiv 0\hspace{0.2cm}(\mathrm{mod}\hspace{0.1cm}q)\}$$ where  integers $\{p_1,...,p_n\}$ are prime to the positive integer $q$. \end{definition} The measures defined in this article can be used in general lattices. But we limit ourselves to preliminary lattices that are useful for the study of lens spaces.
Let $\{v_1,...,v_n\}$ be a basis for $T$. The matrix $A$ whose columns are $ v_1,...,v_n $ is called a generating matrix of $T$. $B$ is another generating matrices of $T$ iff there is a unimodular matrix $U$ such that $A=UB$. An essential parallelepiped of a lattice $T\subset \mathbb{R}^n$ is a parallelepiped $P_T=\{\sum_{i=1}^n a_i v_i | 0 \leq a_i \leq 1,i=1, \cdots,n\}$.
Let $K$ be a cone in $\mathbb{R}^{n} $ whose apex is  the origin. 

\subsection{Harmonic counting measure}
Let $N_{T\cap K}(s)$ be the number of elements in $T\cap K$ with the  $l_1$-norm $s$.
For a   cone $K\subset (\mathbb{R}^+)^n$ set $$F_{T\cap K}(t)=\sum_{s=0}^t\sum_{r=0}^{[\frac{s}{2}]}\left( \begin{array}{cc}
r + n - 2\\
n-2\\
\end{array}\right) N_{T\cap K} (s-2r).$$
\begin{definition} The cone constructed from a set $U\subseteq \mathbb{R}^n $ is the set $\{tx| t \in \mathbb{R}^+ ,\, x\in U\}$. This set is denoted by $C(U)$. \end{definition}
\begin{definition} The harmonic counting measure associated with the lattice T, is a measure $\nu_T$ on the Borel $\sigma$-algebra of $S^{n-1}$ which is defined as \begin{equation}\nu_T(U):=\lim_{t\rightarrow \infty}\frac{F_{T\cap C(U)}(t)}{t^{2n-1}},\end{equation} where $C(U)\subset(\mathbb{R}^+ )^n$. And, it is extended symmetrically to the other Borel subsets. \end {definition}
\hspace{-0.4cm}By lemma 3.6 of \cite{HS},  \begin{equation}
 dim H_{s,{(a_1,...,a_n)}}= \left\{
      \begin{aligned}
        \left( \begin{array}{cc}
r + n - 2\\
n-2\\
\end{array}\right)  \,\,\,\,\,\,\,\,\,\,\,\,\,\,\,\,\,\,\,\,\,\, \parallel{(a_1,...,a_n)}\parallel_{l_1}=s-2r\\
0  \,\,\,\,\,\,\,\,  \,\,\,\,\,\,\,\,\,\,\,\,\,\,\,\,\,\,\,\,\,\,\,\,\,\,\,\,\,\,\,\,\,\,\,\,\,\,\,\,\,\,\,\,\,\,\,\,\,\,\,\,\,\,\,\,\,\,\,\,\,\,\,\,\,\,\,\,\,otherwise\\
 \end{aligned}
    \right.
\end{equation} and it is equal to the number of independent harmonic homogeneous  polynomials of degree $s$ associated with the element $(a_1,...,a_n)$. So the resulting measure is named harmonic-counting measure.

\begin{theorem}\label{12}  $\nu_T$ is a finite measure and its total value is equal to $$\frac{1}{q}2^{1-n}\pi^{1-2n} \omega_{2n-1} \mathrm{Vol}(S^{2n-1}).
$$
\end{theorem}

This is a corollary of Theorem \ref{c} but in the next section we provide another proof for preliminary lattices using the properties of lens spaces.

 In order to study the asymptotic behavior of the function $F_{T\cap K}(t)$, we need the asymptotic behavior of $N_{\mathbb{Z}^n\cap K}(t)$. This is a well-known fact that $\sum_{t=0}^s N_{\mathbb{Z}^n\cap K}(t)\sim\alpha_K s^n$ where $\alpha_K$ is the volume of the intersection of $K$ and the $l_1$-sphere of radius 1 (Ehrhart-Stanley-Minkowski).
This provides a combinatorial approach to a well-known measure $\mu_T$ on the sphere $S^{n-1}$\cite{du}. Precisely  \begin{equation}\mu_T(U)=\lim_{s\rightarrow \infty}\frac{\sum_{t=0}^s N_{\mathbb{Z}^n\cap A^{-1}C(U)}(t)}{s^n}\end{equation} is a finite measure, when  $A$ is the generating matrix of $T$.
\subsection{Lens spaces}

Let $q$ be a positive integer, and let $p_1,...,p_n$ be integers that are prime to $q$.
Let
\begin {equation}
R(\theta)=\left(
            \begin{array}{cc}
              cos\theta & -sin\theta \\
             sin\theta & cos\theta \\
            \end{array}
          \right)\sim e^{i\theta}
\end {equation} and

\begin {equation}
g=R(p_1/q)\oplus \cdots \oplus R(p_n/q).
\end {equation}
Suppose that $G\subset O(2n)$ is the finite cyclic group generated by $g$. If G (as a group of isometries) acts freely on  $S^{2n-1}$, then the manifold  $ S^{2n-1}/G$, denoted by $\mathfrak{L}(p_1,...,p_n;q)$, is called a lens space. Let $spec(M)$ denote the set of eigenvalues of the Laplace-Beltrami operator. $G_1\subseteq G$ implies $spec(S^{2n-1}/G)  \subseteq spec(S^{2n-1}/G_1)$.  In particular $spec(S^{2n-1}/G)  \subseteq spec(S^{2n-1})$. The Laplace-Beltrami eigenvalues of the manifold $S^{2n-1}$ are $ k(k+2n-2), \,  k\in \mathbb{N}\cup\{0\}$  \cite{ik11,ik}. 
\begin{definition}The lens space associated with a lattice $T=\{(a_1,...,a_n) \in \mathbb{Z}^n|\sum_{j=1}^n a_jp_j\equiv 0\hspace{0.2cm}(\mathrm{mod}\hspace{0.1cm}q)\}$ is the space $S^{2n-1}/G$ \cite{HS}.  \end{definition}
A nice relation between lattices and isospectrality is:
\begin{theorem} ( Lauret, Miatello and Rossetti \cite{HS})\label{88}
Two lens spaces $\mathfrak{L}_1=S^{2n-1}/G_1$  and $\mathfrak{L}_2=S^{2n-1}/G_2$ are isospectral iff for the associated lattices $T_1$  and $T_2$, $ B_{l_1 }  (0,k) \cap T_1=B_{l_1 } (0,k) \cap T_2 $ for each $k\in \mathbb{N}$ where $B_{l_1}(0,k) $, is the $l_1$-ball of radius k.  \end{theorem} For a relatively elementary proof see \cite{HO}. 

\begin{theorem} Let $\mathfrak{L}_1$ and $\mathfrak{L}_2$ be two isospectral
 lens spaces with associated lattices $T_1$ and $T_2$. Then $\mu_{T_1}=\mu_{T_2}$\end{theorem}
\begin{proof} It is well-known that  for an arbitrary convex polytope   $\Omega \subset \mathbb{R}^n$ we have $  \lim_{s\rightarrow \infty}\frac{⁡card (\mathbb{Z}^n\cap s\Omega)}{s^n } =\mathrm{Vol}(\Omega)$ \cite{clk}.
 If $A$  is a generating matrix of the lattice $T$ and $K$ is the part of $B_{l_1}(0,1)$ opposite to $U\subseteq S^{n-1}$,  then $\mathrm{card}(T\cap sK)=\mathrm{card} (\mathbb{Z}^n\cap sA^{-1} K)$. Therefore   \begin{equation}\mu_T(U)=  \lim_{s\rightarrow \infty}\frac{⁡\mathrm{card} (\mathbb{Z}^n\cap sA^{-1}K)}{s^n } =\mathrm{Vol}(A^{-1} K)=\mathrm{det}A^{-1} \mathrm{Vol}(K). \end{equation} On the other hand  by Theorem \ref{88}, the value of $\mathrm{det}A_1^{-1} \mathrm{Vol}(B_{l_1}(0,1))$ is equal to $\mathrm{det}A_2^{-1} \mathrm{Vol}(B_{l_1}(0,1))$. Therefore, the determinants of the generating matrices $A_1$ of $T_1$ and $A_2$ of $T_2$ are equal  \cite{HO,L}. Thus these measures are equivalent.\end{proof}
Proof of theorem \ref{12}:\begin{proof}  Let $T$ be a preliminary lattice and  let $\mathfrak{L}$ be its associated lens space. Also, let $K^+$ be the cone of elements of $\mathbb{R}^n$ with positive coordinates. According to \cite{HS}(or \cite{HO}) the number of  independent eigenfunctions of the Laplace-Beltrami operator on a lens space with the eigenvalue  $s(s+(2n-1)-1)$ is equal to $$\sum_{r=0}^{[\frac{s}{2}]}\left( \begin{array}{cc}
             r + n - 2\\
              n-2\\
            \end{array}\right)   N_{T\cap K^+} (s-2r).$$ So $F_{T\cap K^+}(t)$ is equal to the number of independent eigenfunctions with eigenvalues less than  $t(t+2n-2)$. By the Weyl law we have $$
\lim_{x \rightarrow \infty} \frac{N(x)}{x^{\frac{2d-1}{2}}} = (2\pi)^{-(2d-1)} \omega_{2d-1} \mathrm{Vol}(M),
$$ where $N(x)$ denotes the number of eigenvalues less than $x$ and $d$ is  dimension of the manifold $M$. So   \begin{equation}\lim_{t \rightarrow \infty} \frac{F_{T\cap K^+}(t)}{t^{2n-1}}=\lim_{t \rightarrow \infty} \frac{N(t(t+2n-2))}{t^{2n-1}}=\\
\lim_{t \rightarrow \infty} \frac{N(t(t+2n-2))}{(t(t+2n-2))^{\frac{2n-1}{2}}} \end{equation} $$= (2\pi)^{-(2n-1)} \omega_{2n-1} \mathrm{Vol}(S^{2n-1}/G).$$  $S^{2n-1}$ is a $q$-sheeted covering space of $S^{2n-1}/G$ and therefore $\mathrm{Vol}(S^{2n-1}/G)=\frac{1}{q} \mathrm{Vol}(S^{2n-1})$. We have $2^n$ parts for our coordinate, so the achieved number must be multiplied by $2^n$.\end{proof}
Now we compute the value of $\nu_T(U)$ where $U$ is a Borel subset of the sphere $S^{n-1}.$ Let $A$ be the generating matrix of $T$.
\begin{theorem}\label{c}  The value of $\nu_T(U)$ is equal to  
\begin{equation}\lim_{t\rightarrow \infty}\frac{F_{T\cap C(U)}(t)}{t^{2n-1}}=\frac{\boldmath{B}(n-1,n+1)}{(n-2)!2^{n-1}}\mathrm{Vol}(A^{-1}(C(U))\cap B_{l_1}(0,1)), \end{equation} where the beta function is defined as $\boldmath{B}(z,t)=\int_0^1x^{z-1}(1-x)^{t-1} dx$. \end{theorem}

\begin{proof} We have  \begin{equation} \label{a}
F_{T\cap C(U)}(t)=\sum_{s=0}^t\sum_{r=0}^{[\frac{s}{2}]}\left( \begin{array}{cc}
             r + n - 2\\
              n-2\\
            \end{array}\right)   N_{T\cap C(U)} (s-2r), \end{equation} where \begin{equation} \label{ta}\sum_{s=0}^t  N_{T\cap C (U)} (s)=\mathrm{Vol}(A^{-1}(C (U))\cap B_{l_1}(0,1)) t^n +O(t^{n-1})=\alpha t^n+O(t^{n-1}).\end{equation} 
By changing the order of summation in  \ref{a}, we have \\ $$F_{T\cap C(U)}(t)=\sum_{r=0}^{[\frac{t}{2}]}\left(\left( \begin{array}{cc}
r + n - 2\\
n-2\\
\end{array}\right)  \sum_{i=0}^{t-2r} N_{T\cap C(U)} (i)\right).$$ So by \ref{ta},\\
$\frac{1}{t^{2n-1}}\sum_{r=0}^{[\frac{t}{2}]}\left(\left( \begin{array}{cc}
             r + n - 2\\
              n-2\\
            \end{array}\right)  \alpha (t-2r)^n-M(t-2r)^{n-1}\right)\leq $ \\$\frac{1}{t^{2n-1}}\sum_{r=0}^{[\frac{t}{2}]}\left(\left( \begin{array}{cc}
             r + n - 2\\
              n-2\\
            \end{array}\right)  \sum_{i=0}^{t-2r} N_{N_{T\cap C(U)} } (i)\right) \leq \,\,\,\,\,\,\,\,\,\,\,\,\,\,\,\,\,\,\,\,\,\,\,\,\,\,\,\,\,\,\,\,\,\,\,\,\,\,\,\,\,\,\,\,\,\,\,\,\,\,\,\,\,\,\,\,\,\,\,\,\,\,\,\,\,\,\,\,\,\,\,\,\,\,\,\,\,(**)$\\$ \frac{1}{t^{2n-1}} \sum_{r=0}^{[\frac{t}{2}]}\left(\left( \begin{array}{cc}
             r + n - 2\\
              n-2\\
            \end{array}\right) ( \alpha (t-2r)^n+M(t-2r)^{n-1})\right).$\\ Also we have
 \begin{equation}\label{54} \lim_{t\rightarrow\infty}\frac{1}{t}\sum_{i=0}^tg(\frac{i}{t})=\int_0^1g(x)dx,\end{equation} 
 Applying (\ref{54}) the limits of the left and the right parts of $(**)$ are equal to\\ $\frac{\alpha}{(n-2)!2^{n-1}} \int_0^1x^{n-2}(1-x)^n dx$.  So, 
$$\lim_{t\rightarrow \infty}\frac{\sum_{r=0}^{[\frac{t}{2}]}\left(\left( \begin{array}{cc}
	r + n - 2\\
	n-2\\
	\end{array}\right)  \sum_{i=0}^{t-2r} N_{T\cap C(U)} (i)\right)}{t^{2n-1}}=\frac{\alpha}{(n-2)!2^{n-1}} \boldmath{B}(n-1,n+1) .$$ \end{proof}
This shows that  the normalization of $\nu_T$ is a uniform measure  with respect to surface area on each
face of $B_{l_1}(0,1)$.\\
$\boldmath{Remark}\, 1.$ When $\mathfrak{L}$ is the lens space associated with the lattice $T$, the set of independent eigenfunctions of the Laplace-Beltrami operator on $\mathfrak{L}$ associated with the elements of $T \cap C(U) \cap B_{l_1}(0,m(m+2n-2))$ is the same as $F_{C(U)\cap T} (m)$. So, Theorem \ref{c} determines the asymptotic behavior of the cardinality of  eigenfunctions in each direction. Therefore, it  provides more information than Weyl law for Laplace-Beltrami operator in the case of lens spaces.\\
$\boldmath{Remark}\,2.$ Theorem \ref{c}  shows that the number of independent eigenfunctions of the Laplace-Beltrami operator associated with the integral points of $C(U) \cap B_{l_1}(0,t)$ is asymptotically $\frac{B(n-1,n+1)}{(n-2)!2^{n-1} }t^{n-1}$ times the number of lattice points in  $C(U) \cap B_{l_1}(0,t)$.\\
$\boldmath{Remark}\,3.$ Harmonic-counting measures are constant multiples of lattice-counting measures where the constant is an explicit function of dimension of the lattice.\\
$\boldmath{Remark}\,4.$ A harmonic complex polynomial (invariant under the action of  $z_p$ ) can be written uniquely as the summation of harmonic polynomials of the form $Q(|z_1 |^2,\cdots,|z_n |^2 ) z_1^{a_1 }\cdots z_n^{a_n }$  where $(a_1,\cdots,a_n)\in T$ (the lattice associated to the action which is defined the same as the lattice associated to a lens space) and  $Q(w_1,\cdots,w_n )$ is a homogeneous polynomial \cite{HO}. Now let no multiple of $(a_1,\cdots,a_n)$  by an element $0<t<1$ belongs to the lattice $T$. The Theorem \ref{c} asymptotically determines the number of independent homogeneous polynomials $Q$ whose multiplication (after replacing $(w_1,\cdots,w_n)$ by $(|z_1 |^2,\cdots,|z_n |^2 )) $ by an integer power of $z_1^{a_1 }\cdots z_n^{a_n }$ is a harmonic polynomial (Looking at the limit of cones contains the element $(a_1,\cdots,a_n)$).

\end{document}